\documentclass[12pt,leqno]{article}
\usepackage{ulem}
\usepackage[english]{babel}
\usepackage{graphicx}
\usepackage{cancel}
\usepackage{amsmath}
\usepackage{amsthm}
\usepackage{amssymb}
\usepackage{times}
\usepackage{enumerate}
\usepackage{amsmath}
\usepackage{amsfonts}
\usepackage{pdfsync}
\usepackage{cleveref}
\usepackage{autonum}
\usepackage[usenames, dvipsnames]{color}
\usepackage{ulem} 
\usepackage{comment}

\newcommand{\R}{{\mathbb R}}
\newcommand{\N}{{\mathbb N}}
\renewcommand{\L}{\Lambda}

\DeclareMathOperator{\diam}{diam}

\DeclareMathOperator{\dist}{dist}

\makeatletter
\def\author#1{\gdef\autrun{\def\and{\unskip, }#1}\gdef\@author{#1}}

\makeatother

\def\d{\delta}

\newcommand{\Om}{\Omega}
\newcommand{\om}{\omega}

\renewcommand{\d}{{\rm{d}}}

\newcommand{\BBB}{\color{black}}

\newtheorem{theorem}{Theorem}[section]
\newtheorem{corollary}[theorem]{Corollary}
\newtheorem{lemma}[theorem]{Lemma}

\newtheorem*{theorem*}{Theorem}
\usepackage{amsthm}

\theoremstyle{definition}
\newtheorem{definition}[theorem]{Definition}

\newtheorem*{hypothesisH1''}{Hypothesis (S$'_{x_*}$)}
\newtheorem*{hypothesisS'}{Hypothesis (S$'_{x_*}$)}

\newtheorem*{problemP'}{Problem \textbf{(P$'$)}}

\numberwithin{equation}{section}

\begin{document}

\title{Local Lipschitz continuity of the minimizers  of nonuniformly convex functionals    under the  Lower Bounded Slope Condition }
\author{F. Giannetti\thanks{Dipartimento di Matematica e Applicazioni ``R. Caccioppoli'',  Universit\`a degli Studi di Napoli ``Federico II'', via Cintia 80126, Napoli, flavia.giannetti@unina.it}
        \and G. Treu\thanks{{\it Corresponding author,} Dipartimento di Matematica ``Tullio Levi-Civita'', Universit\`a di Padova, via Trieste 63 I-35121, Padova, giulia.treu@unipd.it}}

\maketitle

\begin{abstract}
We prove the local Lipschitz regularity of the minimizers   of
functionals of the form
\[
\mathcal I(u)=\int_\Omega f(\nabla u(x))+g(x)u(x)\,dx\qquad
u\in\phi+W^{1,1}_0(\Omega)
\]
where $g$ is bounded and $\phi$ satisfies the Lower Bounded Slope Condition.
The function $f$ is assumed to be convex but not  uniformly convex everywhere. As byproduct, we also prove the existence  of a locally Lipschitz minimizer  
for a class of functionals of the type above but allowing to the function $f$ to be nonconvex.
\end{abstract}

\begin{keywords}
Lower Bounded Slope Condition, Lipschitz regularity for minimizers, barriers, nonconvex problems.
\end{keywords}

\noindent {\bf AMS Classifications:} 35B51, 49N60, 49-XX

\section{Introduction}

In the present paper   we continue our study in \cite{GT}   of the Lipschitz regularity properties of the minimizers of some functionals in the Calculus of Variations. In particular, we deal with variational integrals
whose Lagrangian  is not uniformly convex with respect to the gradient variable in the entire domain and whose   minimizers satisfy suitable regularity condition  on the boundary.

\noindent Let us start discussing the lack of uniform convexity everywhere.  In the pioneristic paper \cite{ChEv}, the authors faced  the regularity of the minimizers of functionals of the form
\begin{eqnarray}\label{funzCE}
\int_{\Om} f(\nabla
u(x))\,\,d{x},
\end{eqnarray}
with $f$ satisfying  an asymptotic convexity condition, and proved Lipschitz estimates under quadratic growth assumptions. Later on, such estimates were extended to the superquadratic case  in \cite{GM}. Their results  find motivation if one interpret the integral in \eqref{funzCE} as the stored energy of an elastic body: since  no realistic body bears very large deformations, the function $f$  had to be divergent  at infinity not too  wildly so that good estimates  are known. Starting from \cite{ChEv} and \cite{GM}, several authors considered  functionals  uniformly convex not everywhere and  studied the regularity for their minimizers, also in the case of presence of lower order terms.
We refer, for example, to   \cite{LPV}, \cite{FPV1}, \cite{FoGo},  \cite {CGGP1} for smooth functions satisfying standard $p$-growth and to  \cite{CupGuiMas}, \cite{ElMarMas1}, \cite{ElMarMas2} and  \cite {CGGP2} for smooth functions  with nonstandard growth.    The case of non smooth functions satisfying $p$-uniform convexity at infinity was firstly faced in \cite{FF}.

On the other hand, besides some classical and old results by Hilbert, Haar {and} Rado (see   \cite[Chapter 1]{G}),  an increasing interest  on the study of the regularity of minimizers  assuming a boundary  condition is undisputable. 
 This is  due to the fact that suitable boundary assumptions allow to relax, or even drop, both uniform  convexity   and growth conditions. 

\noindent {For example, in \cite{LP} and in \cite{CarGiaLeoPas} the authors  deduced higher integrability  and also $L^\infty$ estimates for minimizers from suitable assumptions on the boundary datum, just assuming a growth condition from below.

\noindent
Moreover, under suitable  boundary datum, Lipschitz and H\"older continuity results for the minimizers have been obtained considering functionals as in \eqref{funzCE} without assuming neither growth conditions from below and from above nor the strict convexity. To this regard we mention, for example, \cite{Cel}, \cite{Cl}, \cite{MT4}, \cite{BousCont}, \cite{BMT}.

\noindent More recently, in \cite{BB}, \cite{FT} and \cite{GT},    global Lipschitz continuity {results have been obtained in case   of functionals of the type}
\begin{eqnarray}
& \displaystyle \int_{\Om} \Big{[}f(\nabla
u(x))+G(x,u(x))\Big{]}\,\,d{x},
\end{eqnarray}
and for  minimizers  satisfying the Bounded Slope Condition (BSC) on the boundary. {We notice that} the validity of (BSC) guarantees  the possibility to construct suitable barriers, yielding pointwise estimates from above and below.  In the three papers above the uniform convexity of $f$ is not assumed in the entire domain. In particular, in \cite{BB} $f$ satisfies a uniform convexity outside of a ball while, in \cite{FT} and \cite{GT},  the uniform convexity is required only in some {suitable} regions of the domain.

\noindent {As in \cite{BB}, in the present paper} we focus  {our} attention on 
 functionals {of the type}
\begin{eqnarray}\label{funzionale}
& \displaystyle {\mathcal I}(u):=\int_{\Om} \Big{[}f(\nabla
u(x))+g(x)u(x)\Big{]}\,\,d{x}.
\end{eqnarray}
{where $f$ is assumed to be} the uniformly convex  outside of a  ball. 

\noindent Note that variational problems of the type 
$$\min_{u\in W_0^{1,2}}\int_{\Omega}F(\nabla u)+\lambda u\, dx$$
with $\Omega\subset \R^2, \lambda\in\R^+$ and
\begin{eqnarray*}
F(\xi)=\begin{cases}
|\xi| \quad |\xi|\le 1\\ \frac{1}{2}|\xi|^2+\frac{1}{2} \quad |\xi|> 1
\end{cases}
\end{eqnarray*}
arise in the study of optimal thin torsion
rods (see for example \cite{ABF}) and can be obviously generalized to the case $\lambda=g(x)$ with $g\in L^{\infty}(\Omega)$.

\noindent On the contrary to \cite{BB} here  we confine ourselves to minimizers satisfying a one-sided Bounded Slope Condition  according to Definition \ref{defbsc}, instead of the Bounded Slope Condition.

\noindent  In this way,  we considerably enlarge the class of boundary functions allowed. Although, under the Lower Bounded Slope Condition {(LBSC)}, there are still significant implications for the regularity of the minimizers as shown  for the first time by Clarke in \cite{Cl} for functionals depending only on the gradient variable.  We would mention also \cite{BC}, where the local regularity {proven in \cite{Cl}} has been extended  to the more general case $g\neq 0$ and $f$ everywhere uniformly convex and the recent paper \cite{BS}, where
the one-sided bounded slope condition has been assumed to study the Lipschitz regularity of  solutions to Cauchy–Dirichlet problems
associated with evolutionary partial differential equations.

In order to clarify our setting and state the main result,  we now describe the integral functional object of our interest in detail.

\noindent  We consider  a function
$f\colon
\R^N\rightarrow \R$   
satisfying  the
following hypotheses
\begin{itemize}
\item[(F1)]  $f(\xi)\ge 0$ and  such that $f(0)=0$;
\item[(F2)]  
${\rm dom}f:=\{x\in \R^N: f(x)\in \R\}
=\R^N$;
\item[(F3)] {there exists a} radius $r>0$  such that  
$$f(\theta\xi+(1-\theta)\zeta)\le\theta f(\xi)+(1-\theta)f(\zeta)-\frac{\epsilon}{2} \theta(1-\theta)|\xi-\zeta|^2$$
 for every $\xi, \zeta\in \R^N$ such that $|\xi|>r+1$ and for every $\theta\in [0,1]$.
 \end{itemize}
\medskip

{Assuming} that $\Om\subset \R^N$ is a $R$-uniformly convex domain in accordance with  Definition \ref{Runiforme}, we consider  integral
functionals as in \eqref{funzionale},
where $g:\Omega\to \R$  is measurable and bounded.

{It is well known that} a  standard application of the Direct Method of the Calculus of Variations ensures the existence of a minimizer of \eqref{funzionale} in the class $\phi+ W^{1,1}(\Omega)$, for any $\phi\in  W^{1,1}(\Omega)$, provided $f$ is convex and
superlinear. This latter condition is satisfied under our assumptions. We point out that our hypotheses on the  function $f$ apply to convex Lagrangians that can be not necessarily differentiable. Moreover, we underline that neither growth conditions from above nor radial structure will be required on $f$.

\noindent We focus our attention to the local Lipchitz regularity of the minimizers as in \cite{Cl} and \cite{BC} and prove the following {two results.}

\begin{theorem}\label{aprioriestimate}
Assume that $f_1\colon \R^N\to \R$ is a $C^{\infty}$ convex function satisfying  hypotheses {\rm (F1)-(F3)} and $f_2\colon \R^N\to \R$ is a $C^{\infty}$ convex function satisfying  hypotheses {\rm (F1)-(F2)} and
$$f_2(\theta\xi+(1-\theta)\zeta)\le\theta f_2(\xi)+(1-\theta)f_2(\zeta)-\frac{\bar\epsilon}{2} \theta(1-\theta)|\xi-\zeta|^2$$
 for every $\xi, \zeta\in \R^N$ {and for every $\theta\in [0,1]$.}
 Suppose that $\Omega$ is a $R$-uniformly convex set and let $g\colon \Omega\to \R$ be a $C^{\infty}(\bar \Omega)$ function.  Fix a function $\phi$ satisfying the (LBSC) and let $u$ be the unique minimizer  of 
 \begin{equation}\label{funzconv}
 \mathcal{J}(u)=\int_{\Omega}f_1(\nabla u)+f_2(\nabla u)+g(x) u\, dx
 \end{equation}
 in the class $\phi+W^{1,1}_0(\Omega)$.
Then, if $u$ is bounded in $\Omega$, it is actually locally Lipschitz  and the following estimate holds
\begin{eqnarray}\label{estimateapp}
 |\nabla u(x)|\le \frac{Q}{\dist(x,\partial \Omega)}\quad \mathrm{for\, a.e} \quad x\in  \Omega,  
\end{eqnarray}
where $Q=Q(\|u\|_\infty, R, {\rm diam}\,\Omega, \epsilon, N, L, \|\phi\|_{1,\infty})$ is independent on $\bar \epsilon$.
\end{theorem}

\begin{theorem}\label{mainth}  Let $f$ be a convex function such that  the hypotheses (F1)-(F3) hold. { Assume that $\Omega$ is a $R$-uniformly convex set} and let $g\colon \Omega\to \R$ be a measurable bounded function.   For a fixed function $\phi$ satisfying the (LBSC), any minimizer $u$ of $\mathcal{I}$ in \eqref{funzionale} is  locally Lipschitz in $\Omega$.
\end{theorem}
{The first of the two results is helpful to prove the second but it is of independent interest. Since its proof is anything but immediate, an entire section has been devoted to it. The functionals considered there have} the Lagrangian with respect to the gradient given by} the sum of two smooth functions $f_1$ and $f_2$, with $f_1$ satisfying assumptions {\rm (F1)-(F3)} and $f_2$ uniformly convex everywhere. With the help of  Theorem 3.1 in \cite{BC} and of Theorem 8.1 in \cite{St}, {we get}  the desired regularity for the unique bounded minimizer of $$\mathcal{J}(u)=\int_{\Omega}f_1(\nabla u)+f_2(\nabla u)+g(x) u\, dx$$ {and} once observed that the minimizer is actually also in the class $C^2(\Omega)$, with some efforts, we prove that its gradient satisfies an estimate independent on the ellipticity constant of $f_2$. The proof of Theorem \ref{mainth} is in Section \ref{sec4} and is splitted in different steps. We start applying the apriori estimate obtained in Theorem \ref{aprioriestimate} to a sequence $\mathcal{I}_k$ of  functionals  approximating $\mathcal{I}$ and, in particular, we prove that it is independent on $k$. {Later on} we show that the sequence of the minimizers $u_k$ of $\mathcal{I}_k$  is actually a minimizing sequence for $\mathcal{I}$ and that converges, up to a subsequence, to a function which inherits the same estimate satisfied by each of them.
The last step {consists in proving} that every minimizer $u$ is locally Lipschitz continuous.

\noindent Even if the global Lipschitz continuity cannot be expected (see the counterexample given in \cite{Cl}),  as a corollary, we  prove that $u$ is H\"older continuous up to the boundary {provided} we add some assumptions on the regularity of  $\partial\Omega$ and on the growth from below of the function $f$ (see Corollary \ref{Clarke}).

\noindent In the last Section of the paper we deal with a class of nonconvex problems. Assuming that the bipolar function $f^{**}$ satisfies conditions (F1)-(F3) and a structure assumption on the set where $f^{**}$ and $f$ do not coincide,  we prove the existence of a minimizer {$u$} of  $\mathcal{I}$ that is still locally Lipschitz, provided the set where $g=0$ has null measure (see Theorem \ref{nonconvex}). For the proof, we use some tools developed in  nonconvex problems of the Calculus of Variations and of Differential Inclusions. {More precisely,} by using a pyramidal construction of the type introduced by Cellina \cite{Cel1} and later developed  in many directions (see for instance \cite{STC}, the book \cite{DM} and references therein), we deduce that the gradient of $u$ belongs almost everywhere to the set where $f^{**}$ and $f$  coincide and this immediately implies that $u$  is also a minimizer of  $\mathcal{I}$. 
We notice that such kind of construction strongly relies on the differentiability almost everywhere of the minimizers of the convexified functionals. For this reason, it is generally added as an hypothesis or it is easily implied by growth assumptions. Here, the minimum $u$ of the convexified functional belongs to the class of locally Lipschitz functions, thanks to Theorem \ref{mainth}, and hence is almost everywhere differentiable.

\noindent We {point out} that, the idea to construct a  minimizer of a nonconvex problem thanks to the preliminary regularity result for the convexified problem has been already present in \cite{FF} but, unlike there, 
we don't assume any growth assumption with respect to the gradient  and we consider functionals depending also on the state variable.  

\noindent As a corollary of Theorem \ref{nonconvex}, we  show the validity of the same result assuming that $g$ doesn't change sign and then the assumption $|\{x\in \Omega:g(x)=0\}|=0$ can be dropped. We  conclude underlying that all the proofs remain valid if we replace the Lower Bounded Slope Condition by the Upper Bounded Slope Condition.

\bigskip

\section{Preliminary results and a comparison principle}\label{Pre}

In this section we give some definitions and remarks.

  \begin{definition}
A function $u\in W^{1,1}(\Om)$ is a {\it minimizer} of
the functional ${\mathcal I}$ if ${\mathcal I}(u)\leq {\mathcal I}(v)$,
for every $v\in u+W^{1,1}_0(\Om)$.
\end{definition}

\begin{definition}[LBSC]\label{defbsc}
A function $\phi:\Omega\subset \R^N \to \R$ satisfies the  {\it lower bounded slope condition} of
rank $M\ge 0$ if for every $\gamma\in\partial\Omega$ there exists
$z_{\gamma}\in \R^N$  such that
\begin{equation}\label{tag:lbsc} 
|z_{\gamma}|\le M\quad {\rm and}\quad \phi(\gamma)+z_{\gamma}\cdot(\gamma'-\gamma)\le\phi(\gamma')\quad \forall\gamma'\in\partial\Omega
\end{equation}
\end{definition}
Recalling that the polar  of a convex function $f:\R^N \to R$ is the function
$f^*\colon \R^N\rightarrow[-\infty,+\infty]$  defined by
$$f^*(\xi):=\sup_{x\in \R^N}\{x\cdot\xi-f(x)\}, \quad \forall \xi\in \R^N$$
(see {for example} \cite{ET}), we state the following Lemma  (see Lemma 3.1 in \cite{GT}.)
\begin{lemma}\label{lemma:superlinearity} Assume that $f$ is  convex, ${\rm dom}f=\R^N$, $f(\xi)\ge 0$ and $f(0)=0$. Then
$f$ is superlinear if and only if  ${\rm dom}f^*=\R^N$.
\end{lemma}
 Let us fix a Lipschitz function $\phi:\Omega\to \R$ with Lipschitz constant $M$, a point $x_0\in \R^N$
and  {consider}  the  functions
\begin{eqnarray}\label{ausiliarie}
& \displaystyle
\om_{\alpha}(x):=\frac{N}{\alpha} f^*\Big{(}\alpha\frac{x-x_0}{N}\Big{)}.
\end{eqnarray}
introduced in \cite{C3}. They  are Lipschitz continuous in $\Omega$ with a constant $K$ which can be assumed, without loss of generality, greater than $M$.
The following   adapted version of the comparison principle in Theorem 2.4 of \cite{FT} (see also Theorem 2.4 in \cite{GT}) holds. 
\begin{theorem}\label{thm:comparison1}
Assume that $f$, $\{f_k\}_k$  are convex and superlinear functions such that $f-1\le f_k\le f$  in $ \R^N$
and suppose that $g$ is a bounded function such that $\|g\|_{\infty}\le \alpha$.
For a fixed bounded function $\phi\in W^{1,1}(\Omega)$, let  $u_k$ be a minimizer of 
$$\mathcal{J}_k(u):=\int_{\Omega}f_k(\xi)+ g(x)u(x)\, dx$$ in $\phi+W_0^{1,1}(\Omega)$. 
Then there exist $c_1, c_2\in\R$ such that 
\begin{eqnarray}\label{tesi:confronto}
\omega_{\alpha}+c_1\le u_k\le \omega_{-\alpha}+c_2\quad\text{a.e. in}\ \Omega.
\end{eqnarray}

\end{theorem}
\begin{proof}
Let us define for every $k\in \N$ and for a fixed $x_0\in \R^N$ the following functions
\begin{eqnarray}
& \displaystyle
\om_{\alpha}^k(x):=\frac{N}{\alpha} f_k^*\Big{(}\alpha\frac{x-x_0}{N}\Big{)}
\end{eqnarray}
and
\begin{eqnarray}
& \displaystyle
\tilde \om_{\alpha}(x):=\frac{N}{\alpha} (f-1)^*\Big{(}\alpha\frac{x-x_0}{N}\Big{)}. 
\end{eqnarray}
 It holds that For every $k\in \N$and for every $x\in \R^N$, it holds that 
\begin{eqnarray}\label{relazioniomega}
    \om_{\alpha}(x)\le \om_{\alpha}^k(x)\le \tilde\om_{\alpha}(x).
    \end{eqnarray}
Set $c_1=\inf_{\partial \Omega}\{\phi-\tilde\om_{\alpha}\}$, then for every $k\in \N$, we obviously have
$$\om_{\alpha}^k(x)+c_1\le \phi(x) \quad{\rm on } \, \partial \Omega$$  and hence
Theorem 2.4 in \cite{FT} yields for every $k\in \N$
$$\om_{\alpha}^k(x)+c_1\le u_k(x) \quad {\rm a.e \,\, in}\,  \Omega.$$ Hence the first inequality at \eqref{tesi:confronto} follows from \eqref{relazioniomega}. An analogous argument gives the existence of $c_2$ such that
$\om_{-\alpha}(x)+c_2\ge u_k(x)$ a.e. in $\Omega$.
\end{proof}
With the result in Theorem \ref{thm:comparison1} in our hands, we are in position to construct the barriers for every minimizer in the desired class. More precisely, we can use a one-sided version of Theorem 4.5 in \cite{GT}. We first need to recall  the following

\begin{definition}\label{Runiforme}
An  open bounded subset $\Omega$ of $\R^N$ is  $R$-{\it uniformly convex}, $R>0$, if for every $\gamma\in\partial\Om$
there exists a vector $b_{\gamma}\in \R^N$, with $|b_{\gamma}|=1$, such that
\begin{eqnarray}
& \displaystyle R\, b_{\gamma}\cdot(\gamma'-\gamma)\geq
\frac12|\gamma'-\gamma|^2, \quad \forall \gamma'\in \partial \Omega.\label{unif-cvx}
\end{eqnarray}
\end{definition}

\begin{theorem}\label{thm:barriere}
  Assume that $f\colon \R^N\to \R$ is convex and satisfies hypotheses {\rm (F1)-(F3)}.  Let $\Omega$ be an open bounded $R$-uniformly convex set and $g\colon \Omega\to \R$ be a measurable bounded function.

\noindent For every fixed function $\phi:\Omega\rightarrow\R$ satisfying the (LBSC),  there exists
   $\ell :\overline\Omega\rightarrow\R$,
Lipschitz of rank $L=L(R,f,\phi)$, such that
\[\ell(\gamma)=\phi(\gamma)\quad\textrm{ for every }\gamma
\in\partial\Omega\] and
\[\ell(x)\le u(x)\]
for almost every $x
\in\Omega{\BBB }$ and for every minimum $u$ in $\phi+W_0^{1,1}(\Omega)$.
\end{theorem}
We conclude the section proving a useful uniform approximation result on $\R^N$. 
\begin{theorem}\label{approximation}
Let $f:\R^N\to \R$ be a convex function such that, for some $\mu>0$, satisfies 
\begin{eqnarray}\label{uc}f(\theta\xi+(1-\theta)\zeta)\le\theta f(\xi)+(1-\theta)f(\zeta)-\frac{\mu}{2} \theta(1-\theta)|\xi-\zeta|^2
\end{eqnarray}
 for every $\xi, \zeta\in \R^N$  and for every $\theta\in [0,1]$. There exists a sequence of $C^{\infty}$ functions satysfing a  condition analogous to \eqref{uc} converging to $f$ uniformly on $\R^N$.
\end{theorem}

\begin{proof}
 In order to prove our aim, it will be sufficient to prove that for every $\epsilon>0$, there exists a $C^{\infty}$ convex function $g$ such that
$$f-2\epsilon\le g\le f$$ and satisfying \eqref{uc} with $\mu'$ in place of $\mu$.
 We shall follow the ideas in Theorem 1.2 of \cite{Azagra}.
 
Let us start considering the  nondecreasing sequence $\{f_k\}_k$ of convex, smooth
functions  converging uniformly  to $f$  on compact sets constructed  in Lemma 7.4 of \cite{BB}. Such functions satisfy the property in \eqref{uc} for a positive constant $\mu'<\mu$.

\noindent Let us choose $k_n\in \N$ sufficiently large to have  that the  function $f_{k_n}$ satisfies 
$$f_{k_n}\ge f-\frac{\epsilon}{2^{n+1}} \quad {\rm on \,the\, ball}\, B_n.$$  We obviously have that
$$f_{k_n}\le f  \quad {\rm on } \, \R^N.$$
Setting
$$h_n=f_{k_n}-\sum_{k=0}^{n-1}\frac{\epsilon}{2^{k}}-\frac{\epsilon}{2^{n+1}},$$
we can check that
$$h_n\ge f-\sum_{k=0}^{n}\frac{\epsilon}{2^{k}} \quad {\rm on} \quad B_n$$
and
$$h_n\le f-\sum_{k=0}^{n-1}\frac{\epsilon}{2^{k}}-\frac{\epsilon}{2^{n+1}}\quad {\rm on} \quad \R^N.$$

\noindent As in the proof of Theorem 1.2 in \cite{Azagra}, we define $g_1=h_1$ and get, for every $n\ge 2$, the existence of a $C^{\infty}$ convex function $g_n$ 
   such that 
$$\max\{g_{n-1},h_n\}\le g_n\le \max\{g_{n-1},h_n\} +\frac{\epsilon}{10^n} \quad {\rm on } \quad \R^N$$
Moreover, reasoning similarly as in the proof of Proposition 2.2 (ix)  in \cite{Azagra}, we can deduce that $g_n$ satisfies the property in \eqref{uc}   with  $\mu'$ in place of $\mu$.
Finally, going in the deep of the proof of Theorem 1.2 in \cite{Azagra}, we verify that
$$g_n=g_{n+1} \quad {\rm on} \quad B_n$$
and
$$f-\sum_{k=0}^{n}\frac{\epsilon}{2^k}\le g_n\le f-\frac{\epsilon}{2}+ \sum_{k=2}^{n}\frac{\epsilon}{10^k}\quad {\rm on} \quad B_n.$$
 Therefore the pointwise limit of the sequence $\{g_n\}_n$ is a $C^{\infty}$ function $g$ such that
$$f-2\epsilon\le g\le f$$ 
and satisfying the property \eqref{uc} with  $\mu'$ in place of $\mu$.

\end{proof}

In order to face the nonconvex case in Section \ref{nccase}, we  recall that  the convexification $f^{**}$ of a function $f\colon \R^N\to \bar\R$ is the greatest  convex function which is less than $f$ on $\R^N$.

\section{Proof of Theorem \ref{aprioriestimate}}

This section is devoted to the proof of Theorem \ref{aprioriestimate}. The inequality  \eqref{estimateapp} will be used as an apriori estimate in the proof of Theorem \ref{mainth}. 

\begin{proof}
 We start observing that the  minimizer in the class $\phi+W^{1,1}_0(\Omega)$ of the functional \eqref{funzconv} is unique, locally Lipschitz and belonging to $C^2(\Omega)$  (see Theorem 8.1 in \cite{St} and Theorem 3.1 in \cite{BC}).

\noindent Consider $\ell(x)$ the 
 convex  and Lipschitz function of rank  $L$ given by Theorem \ref{thm:barriere}, then we get  
 $$ u(x)\ge \ell(x)\ge K\cdot (x-\gamma)+\phi (\gamma) \quad \forall x\in \Omega, \forall \gamma \in \partial \Omega$$
 where $K\in \partial \ell(\gamma), |K|\le L$. As in \cite{BC}, consider $\lambda\in [1/2,1)$ and for a  fixed  point $z\in \partial \Omega$, define $$\Omega_{\lambda}:= \lambda (\Omega-z)+z$$
and
$$ u_{\lambda}(y):= \lambda  u((y-z)/\lambda+z).$$
Then $ u_{\lambda}\in \phi_{\lambda}+W^{1,1}_0(\Omega_{\lambda})\cap C^2(\Omega_{\lambda})$ with
$$\phi_{\lambda}(y):= \lambda \phi((y-z)/\lambda+z).$$
Arguing as in the proof of Lemma 2.6 in \cite{Cl}, we set $\bar q:=  L \diam \Omega+ \|\phi\|_{L^{\infty}(\partial \Omega)}$ and have for $q>\bar q+1$
\begin{eqnarray}\label{relsull}
  u_{\lambda}\le  u +\left(q-\frac{1}{n}\right)(1-\lambda)\,\text{on }\partial\Omega_{\lambda}. \end{eqnarray}
  Let us  prove that,  for a $\lambda$ sufficiently close to $1$ and $q$ sufficiently large, it is also 
\begin{eqnarray}\label{relsullafrontiera}
  u_{\lambda}\le  u+\left(q-\frac{1}{n}\right)(1-\lambda) \,\, \text{a.e. in }\Omega_{\lambda}. \end{eqnarray}

\noindent To this aim, consider the set $S_{\lambda,q, n}=\{y\in \Omega_\lambda:  u_{\lambda}(y)> 
 u(y)+\left(q-\frac{1}{n}\right)(1-\lambda)\}$ and prove that its measure is equal to $0$ independentently on $\lambda$ sufficiently close to $1$ and $q$ sufficiently large.

\noindent Let  $w:=\min \{ u+\left(q-\frac{1}{n}\right)(1-\lambda), u_{\lambda}\}$ and observe that, thanks to \eqref{relsull}, $w\in \phi_{\lambda}+W^{1,1}_0(\Omega_{\lambda})\cap C^2(\Omega_{\lambda})$ so that
$$w^{\lambda}(x):=\frac{1}{\lambda}w(\lambda (x-z)+z)$$
belongs to $\phi+W^{1,1}_0(\Omega)\cap C^2(\Omega)$. Using now the minimality of $u$ we can easily deduce that $$\mathcal{J}\left(u\right)\le \mathcal{J}(v)$$ with
$v:=\theta w^{\lambda}+(1-\theta) u\in \phi+W^{1,1}_0(\Omega)\cap C^2(\Omega)$, $\theta\in [0,1]$,  that is
\begin{align}
&\int_{\Omega}f_1(\nabla  u)+f_2(\nabla u)+g(x)  u(x)\, dx\\ &\le \int_{\Omega}f_1(\theta \nabla w^{\lambda}+(1-\theta)\nabla   u)+f_2(\theta \nabla w^{\lambda}+(1-\theta)\nabla   u)+g(x) (\theta  w^{\lambda}+(1-\theta)  u)\, dx
\end{align}
and with a change of variable
\begin{eqnarray*}
&&\int_{\Omega_\lambda}f_1(\nabla  u_{\lambda})+f_2(\nabla  u_{\lambda})+g((y-z)/\lambda+z)  u((y-z)/\lambda+z) 
\, dy\cr\cr
&&\le \int_{\Omega_\lambda}f_1(\theta \nabla w+(1-\theta)\nabla  u_{\lambda})+f_2(\theta \nabla w+(1-\theta)\nabla  u_{\lambda})\cr\cr&&+g((y-z)/\lambda+z)) (\theta  w^\lambda((y-z)/\lambda+z)) +(1-\theta)   u((y-z)/\lambda+z))\, dy\cr\cr&&
\le \int_{\Omega_\lambda}f_1(\theta \nabla w+(1-\theta)\nabla  
 u_{\lambda})+f_2(\theta \nabla w+(1-\theta)\nabla  
 u_{\lambda})\cr\cr&&+g((y-z)/\lambda+z)) \left(\theta  \frac{w}{\lambda} +(1-\theta) \frac{ 
 u_{\lambda}}{\lambda}\right) \, dy.
\end{eqnarray*}
Observing that $w= u_{\lambda}$   on $\Omega_\lambda\setminus S_{\lambda,q,n}$ and $w=  u+\left(q-\frac{1}{n}\right)(1-\lambda)$ on $S_{\lambda,q,n}$, previous inequality gives

\begin{eqnarray}
&&\int_{S_{\lambda,q,n}}f_1(\nabla  u_{\lambda})+f_2(\nabla  u_{\lambda})+g((y-z)/\lambda+z)  \frac{u_{\lambda}}{\lambda}\, dy\cr\cr
&&\le \int_{S_{\lambda,q,n}}f_1(\theta \nabla u+(1-\theta)\nabla  
 u_{\lambda})+f_2(\theta \nabla u+(1-\theta)\nabla  
 u_{\lambda})\cr\cr&&+g((y-z)/\lambda+z)) \left(\theta  \frac{u+\left(q-\frac{1}{n}\right)(1-\lambda)}{\lambda} +(1-\theta) \frac{ 
 u_{\lambda}}{\lambda}\right) \, dy.
\end{eqnarray}
and therefore
\begin{align}\label{prima}
&\int_{S_{\lambda,q,n}}\!\!\!\!\!f_1(\nabla  u_{\lambda})-f_1(\theta \nabla u+(1-\theta)\nabla  
 u_{\lambda})\\&+f_2(\nabla  u_{\lambda})-f_2(\theta \nabla u+(1-\theta)\nabla  
 u_{\lambda})\, dy\\
&\le \int_{S_{\lambda,q,n}}\!\!\!\!\!g((y-z)/\lambda+z)) \left(\theta  \frac{u}{\lambda} +(1-\theta) \frac{ 
 u_{\lambda}}{\lambda}\right)-g((y-z)/\lambda+z)  \frac{u_{\lambda}}{\lambda} \\&+g((y-z)/\lambda+z))\theta  \frac{\left(q-\frac{1}{n}\right)(1-\lambda)}{\lambda}\, dy
 \\&=
 \int_{S_{\lambda,q,n}}\!\!\!\!\!g((y-z)/\lambda+z)) \frac{\theta}{\lambda}\left( u - 
 u_{\lambda}\right)\\& +g((y-z)/\lambda+z))\theta  \frac{\left(q-\frac{1}{n}\right)(1-\lambda)}{\lambda}\, dy
\end{align}
Now consider 
 $W:=\max \{ u, u_{\lambda}-\left(q-\frac{1}{n}\right)(1-\lambda)\}$. {We can} observe that $W= u$ on $\Omega \setminus S_{\lambda,q,n}$ and that $W\in \phi+W^{1,1}_0(\Omega)\cap C^2(\Omega)$ thanks to \eqref{relsull}.
 For $\theta\in[0,1]$ and 
 $V:=\theta W+(1-\theta) u\in \phi+W^{1,1}_0(\Omega)\cap C^2(\Omega)$, using the minimality of $  u$,   we have $\mathcal{J}( u)\le \mathcal{J}(V)$ that is 
 \begin{eqnarray*}
&&\int_{\Omega}f_1(\nabla  u)+f_2(\nabla  u)+g(x)  u(x)\, dx \le \int_{\Omega}f_1(\theta \nabla W+(1-\theta)\nabla   u)\cr\cr &&+f_2(\theta \nabla W+(1-\theta)\nabla   u)+g(x) (\theta  W+(1-\theta)  u)\, dx
\end{eqnarray*}
and hence
\begin{eqnarray}\label{seconda}
&&\int_{S_{\lambda,q,n}}f_1(\nabla  u)-f_1(\theta \nabla u_{\lambda}+(1-\theta)\nabla  u)+ \cr\cr
&&+f_2(\nabla  u)-f_2(\theta \nabla u_{\lambda}+(1-\theta)\nabla  u)\, dy\cr\cr
&&\le
\int_{S_{\lambda,q,n}}g(y) \left(\theta   u_{\lambda} +(1-\theta)  u\right)-g(y)  u(y) \cr\cr&&
-g(y)\theta \left(q-\frac{1}{n}\right)(1-\lambda)\, dy\cr\cr&&
=
\int_{S_{\lambda,q,n}}g(y) \theta  ( u_{\lambda} -u(y)) -g(y)\theta \left(q-\frac{1}{n}\right)(1-\lambda)\, dy
\end{eqnarray}
Now, summing up  inequalities \eqref{prima} and \eqref{seconda}, we get
 \begin{eqnarray}
&&\int_{S_{\lambda,q,n}}(1-\theta)f_1(\nabla  u_{\lambda})+\theta f_1(\nabla  u)-f_1(\theta \nabla  u+(1-\theta) \nabla  u_{\lambda})\cr\cr&&+\theta f_1(\nabla  u_{\lambda})+(1-\theta)f_1(\nabla  u)-f_1(\theta \nabla  u_{\lambda}+
(1-\theta)
\nabla  u)\,dy\cr\cr&&+
\int_{S_{\lambda,q,n}}(1-\theta)f_2(\nabla  u_{\lambda})+\theta f_2(\nabla  u)-f_2(\theta \nabla  u+(1-\theta) \nabla  u_{\lambda})\cr\cr&&+\theta f_2(\nabla  u_{\lambda})+(1-\theta)f_2(\nabla  u)-f_2(\theta \nabla  u_{\lambda}+
(1-\theta)
\nabla  u)\,dy
\cr\cr&&
\le \int_{S_{\lambda,q,n}}\theta\left[\frac{1}{\lambda} g\left(\frac{y-z}{\lambda}+z\right)-g(y)\right](u-u_{\lambda})\,dy\cr\cr&&
+ \int_{S_{\lambda,q,n}}\theta\left[\frac{1}{\lambda} g\left(\frac{y-z}{\lambda}+z\right)-g(y)\right]\left(q-\frac{1}{n}\right)(1-\lambda)\, dy
\end{eqnarray}
and, observing that the second integral in the left hand side is positive by  convexity, we get
\begin{eqnarray*}
&&\int_{S_{\lambda,q,n}}(1-\theta)f_1(\nabla  u_{ \lambda})+\theta f_1(\nabla  u)-f_1(\theta \nabla  u+(1-\theta) \nabla  u_{\lambda})\cr\cr&&+\theta f_1(\nabla  u_{\lambda})+(1-\theta)f_1(\nabla  u)-f_1(\theta \nabla  u_{\lambda}+
(1-\theta) \nabla  u)\, dy\cr\cr&&
\le \int_{S_{\lambda,q,n}}\theta\left[\frac{1}{\lambda} g\left(\frac{y-z}{\lambda}+z\right)-g(y)\right]\left(u+\left(q-\frac{1}{n}\right)(1-\lambda)-u_{\lambda}\right)\,dy.
\end{eqnarray*}
Define now 
$$S_{\lambda,q, n}(r)\!\!=\!\!\{x\in \Omega_\lambda:  u_{\lambda}(x)> 
 u(x)+\left(q-\frac{1}{n}\right)(1-\lambda), |\nabla u|>r+1 \}.$$
Using  {assumption (F3)} to estimate from below the left hand side of previous inequality and dividing for $\theta$, we  have
\begin{eqnarray*}
&&{\epsilon\, (1-\theta)} \int_{S_{\lambda,q,n}(r)}|\nabla  u_{\lambda}-\nabla  u|^2dy \cr\cr&&\,\,\,\,\,\,\le
 \!\!\int_{S_{\lambda,q,n}}\left(\frac{1}{\lambda}g\left(\frac{y-z}{\lambda}+z\right)-g(y)\right)\left(u+\left(q-\frac{1}{n}\right)(1-\lambda)-u_{\lambda}\right)\, dy
\end{eqnarray*}
{and hence, letting $\theta\to 0$},

\begin{eqnarray}\label{finale}
&&{\epsilon} \int_{S_{\lambda,q,n}(r)}|\nabla  u_{\lambda}-\nabla  u|^2dy \cr\cr&&\,\,\,\,\,\,\le
\!\!\int_{S_{\lambda,q,n}}\left(\frac{1}{\lambda}g\left(\frac{y-z}{\lambda}+z\right)-g(y)\right)\left(u+\left(q-\frac{1}{n}\right)(1-\lambda)-u_{\lambda}\right)\, dy
\end{eqnarray}

In order to pass to polar coordinates and to simplify the calculations, we first set $u+\left(q-\frac{1}{n}\right)(1-\lambda)-u_{\lambda}=0$ outside $S_{\lambda,q,n}$ and reason on  the following integral
\begin{eqnarray*}
\int_{\Omega}\left(\frac{1}{\lambda}g\left(\frac{y-z}{\lambda}+z\right)-g(y)\right)\left(u+\left(q-\frac{1}{n}\right)(1-\lambda)-u_{\lambda}\right)^-\, dy.
\end{eqnarray*}
Consider a change of variable centered in $z$
such that $\rho=|y-z|$, $y=\rho h(\theta)$.
Hence 
\begin{align}\label{cambiovar}
&\int_{\Omega}\left(\frac{1}{\lambda}g\left(\frac{y-z}{\lambda}+z\right)-g(y)\right)\left(u+\left(q-\frac{1}{n}\right)(1-\lambda)-u_{\lambda}\right)^-\, dy\\
&=\int_0^{\pi}\d\theta \int_{0}^{\rho(\theta)}
\left(\frac{1}{\lambda}g\left(h(\theta)\left(\rho+\frac{1-\lambda}{\lambda}\rho\right)\right)-g(\rho h(\theta))\right)\cdot \\&\left(\!\!u(\rho h(\theta))\!\!+\!\!\left(q-\frac{1}{n}\right)(1-\lambda)-\lambda u\left(h(\theta)(\rho+\frac{1-\lambda}{\lambda}\rho)\right)\right)^-\!\!\!\!\rho^{N-1}|J_{h}|\, d\rho.
\end{align}
where $J_{h}$ is the jacobian determinant of $h$. 
Setting 
$$H(\rho)=\int_{\rho}^{\rho+\frac{1-\lambda}{\lambda}\rho}g(sh(\theta))\ ds$$ 
and 
$$G(\rho)=H(\rho)\rho^{N-1}-(N-1)\int_0^\rho \rho^{N-2}H(s)\, ds$$
so that
$$H'(\rho)=\frac{1}{\lambda}g\left(h(\theta)\left(\rho+\frac{1-\lambda}{\lambda}\rho\right)\right)-g(\rho h(\theta))$$
and 
$$G'(\rho)=H'(\rho)\rho^{N-1},$$
the integral in \eqref{cambiovar} becomes by integration by parts the following
\begin{eqnarray*}
\int_0^{\pi}\d\theta \int_{0}^{\rho(\theta)}G(\rho)\left[\nabla \left(u\left(h(\theta)\left(\rho+\frac{1-\lambda}{\lambda}\rho\right)\right) - u(\rho h(\theta))\right)^-\right]|J_h|\, d\rho.
\end{eqnarray*}
Recalling now inequality \eqref{finale}, we obtain
\begin{eqnarray*}
&&{\epsilon}\int_{S_{\lambda,q,n}(r)}|\nabla  u_{\lambda}-\nabla  u|^2dy \cr\cr&&\le
\int_0^{\pi}\d\theta \int_{0}^{\rho(\theta)}G(\rho)\left[\nabla \left(u \left(h(\theta)\left(\rho+\frac{1-\lambda}{\lambda}\rho\right)\right) - u(\rho h(\theta))\right)^-\right]|J_h|\, d\rho
\end{eqnarray*}
and dividing for $(1-\lambda)^2$, we have
\begin{eqnarray}\label{completa}
&&{\epsilon}\int_{S_{\lambda,q,n}(r)}\frac{|\nabla  u_{\lambda}-\nabla  u|^2}{(1-\lambda)^2}dy \cr\cr &&\le
\int_0^{\pi}\d\theta \int_{0}^{\rho(\theta)}\frac{G(\rho)}{(1-\lambda)}\frac{\left[\nabla \left(u \left(h(\theta)\left(\rho+\frac{1-\lambda}{\lambda}\rho\right)\right) - u(\rho h(\theta))\right)^-\right]}{(1-\lambda)}|J_h|\, d\rho.
\end{eqnarray}
If we observe that the integral in the left hand side can be written as
$$\int_{S_{\lambda,q,n}(r)}\frac{|\nabla u(\frac{y-z}{\lambda}+z)-\nabla u(y)|^2}{(1-\lambda)^2}\, dy=\int_{S_{\lambda,q,n}(r)}\frac{1}{\lambda^2}\frac{|\nabla u(y+\frac{1-\lambda}{\lambda}(y-z))-\nabla u(y)|^2}{\frac{(1-\lambda)^2}{\lambda^2}}\, dy,$$
by formula \eqref{completa}, we deduce that 

\begin{eqnarray}\label{rh}
{\epsilon}\lim_{\lambda\to 1}\int_{S_{\lambda,q,n}(r)}\frac{1}{\lambda^2}\frac{|\nabla u(y+\frac{1-\lambda}{\lambda}(y-z))-\nabla u(y)|^2}{\frac{(1-\lambda)^2}{\lambda^2}}\, dy
\end{eqnarray}
\begin{align}
&\le\lim_{\lambda\to 1}\int_0^{\pi}\!\! \int_{0}^{\rho(\theta)}\frac{G(\rho)}{(1-\lambda)}\frac{\left[\nabla \left(u \left(h(\theta)\left(\rho+\frac{1-\lambda}{\lambda}\rho\right)\right) - u(\rho h(\theta))\right)^-\right]}{(1-\lambda)}|J_h|\, d\rho\,\d\theta\\
&\le\lim_{\lambda\to 1}\int_0^{\pi}\!\!\int_{0}^{\rho(\theta)}\!\!\!\!\frac{G(\rho)}{(1-\lambda)}\frac{\nabla \left[\left(u \left(h(\theta)\left(\rho+\frac{1-\lambda}{\lambda}\rho\right)\right) - u(\rho h(\theta))\right)\chi_{S_{\lambda, q, n}}\right]}{(1-\lambda)}|J_h|\, d\rho \, d\theta 
\end{align}

In order to calculate the limits in the previous inequality, we first need to know the limits of the sets $S_{\lambda,q,n}(r)$ and $S_{\lambda,q,n}$. To this aim, observe that
$$u_{\lambda}(x)> 
 u(x)+\left(q-\frac{1}{n}\right)(1-\lambda)$$
 is equivalent to 
 $$\frac{\lambda u\left(\frac{x-z}{\lambda}+z\right)-u(x)}{1-\lambda}>q-\frac{1}{n}$$
 that is
 $$\frac{\lambda u\left(\frac{x-z}{\lambda}+z\right)-\lambda u(x)-(1-\lambda )u(x)}{1-\lambda}>q-\frac{1}{n}$$
 and also
  $$\frac{\lambda u\left(\frac{x-z}{\lambda}+z\right)-\lambda u\left(\frac{x+(1-\lambda)(z-x)-z}{\lambda}+z\right)-(1-\lambda )u(x)}{1-\lambda}>q-\frac{1}{n}$$
 i.e.
 \begin{eqnarray}\label{primafraz}\frac{u\left(\frac{x-z}{\lambda}+z\right)-u(x)}{\frac{1-\lambda}{\lambda}}-u(x)>q-\frac{1}{n}
 \end{eqnarray}
so that, if we pass to the limit as $\lambda\to 1$,
 we get 
\begin{eqnarray}
\langle \nabla u(x), x-z\rangle -u(x)\ge q-\frac{1}{n}
\end{eqnarray}
which yields, for $q>(r+1)\diam \Omega+\|u\|_{\infty}+1$,
 $$|\nabla u|> r+1$$
 It follows that
 \begin{eqnarray*}
 &&S_{q,n}^-:=\{\langle \nabla u(x), x-z\rangle -u(x)> q-\frac{1}{n}\}\subseteq\lim_{\lambda\to 1}S_{\lambda,q,n}(r)\cr\cr&&\subseteq\lim_{\lambda\to 1}S_{\lambda,q,n}=S\subseteq \{\langle \nabla u(x), x-z\rangle -u(x)\ge q-\frac{1}{n}\}=:S_{q,n}^+
 \end{eqnarray*}
 Now {let's switch on the right hand side of inequality \eqref{rh} observing} that 
\begin{eqnarray*}
&&\frac{G(\rho)}{(1-\lambda)}=\frac{H(\rho)\rho^{N-1}-(N-1)\int_0^\rho \rho^{N-2}H(s)\, ds}{(1-\lambda)}\cr\cr&&=
\frac{\rho^{N-1}\int_{\rho}^{\rho+\frac{1-\lambda}{\lambda}\rho}g(sh(\theta))\ ds-(N-1)\rho^{N-2}\int_0^\rho \int_{s}^{s+\frac{1-\lambda}{\lambda}s}g(th(\theta))\ dt}{(1-\lambda)}\cr\cr&&=
\frac{\rho^N}{\lambda}\left[\frac{\lambda}{(1-\lambda)\rho}\int_{\rho}^{\rho+\frac{1-\lambda}{\lambda}\rho}g(sh(\theta))\ ds\right]\cr\cr&&-(N-1)\rho^{N-2}\int_0^\rho \frac{s}{\lambda} \left(\frac{\lambda}{(1-\lambda)s}\int_{s}^{s+\frac{1-\lambda}{\lambda}s}g(th(\theta))\ dt\right)\, ds,
\end{eqnarray*}
and hence 
\begin{eqnarray*}
&&\lim_{\lambda\to 1}\int_0^{\pi}\d\theta \int_{0}^{\rho(\theta)}\frac{G(\rho)}{(1-\lambda)}\frac{\nabla \left[\left(u \left(h(\theta)\left(\rho+\frac{1-\lambda}{\lambda}\rho\right)\right) - u(\rho h(\theta))\right)\chi_{S_{\lambda, q, n}}\right]}{(1-\lambda)}|J_h|\, d\rho\cr\cr
&&=\int_0^{\pi}\!\!\d\theta \int_{0}^{\rho(\theta)}\!\!\left[\rho^N g(\rho h(\theta))-(N-1)\rho^{N-2}\int_0^\rho s g(sh(\theta))\, ds\right] \!\!\langle D^2u(\rho h(\theta))\chi_{S}, \rho h(\theta)\rangle\,|J_h|\, d\rho
\cr\cr
&&\le
\int_0^{\pi}\!\!\d\theta \int_{0}^{\rho(\theta)}\!\!\left|\rho^N g(\rho h(\theta))-(N-1)\rho^{N-2}\int_0^\rho s g(sh(\theta))\, ds\right| \cdot|\langle D^2u(\rho h(\theta))\chi_{S}, \rho h(\theta)\rangle|\,|J_h|\, d\rho\cr\cr&&\le
\int_0^{\pi}\!\!\d\theta \int_{0}^{\rho(\theta)}\!\!\left[\rho^N\|g\|_{\infty}+(N-1)\rho^{N-2}\int_{0}^{\rho}s\|g\|_{\infty}\, ds\right]\cdot|\langle D^2u(\rho h(\theta))\chi_{S}, \rho h(\theta)\rangle|\, |J_h|\,d\rho\cr\cr&&\le
\int_0^{\pi}\!\!\d\theta \int_{0}^{\rho(\theta)}\frac{N+1}{2}\rho\|g\|_{\infty}\cdot|\langle D^2u(\rho h(\theta))\chi_{S}, \rho h(\theta)\rangle||J_h|\, \rho^{N-1}d\rho
\cr\cr&& \le \frac{N+1}{2}\|g\|_{\infty} \int_{S_{ q, n}^+}|y-z|\cdot |\langle D^2u(y), (y-z)\rangle|\, dy\cr\cr&&\le
\frac{N+1}{2}\|g\|_{\infty} {\rm diam} \,\Omega\int_{S_{ q, n}^+} |\langle D^2u(y), (y-z)\rangle|\, dy.
\end{eqnarray*}
Then formula \eqref{rh} yields
 \begin{eqnarray*}
 &&{\epsilon}\int_{S_{q,n}^-}|\langle D^2u(y), (y-z)\rangle|^2\, dy\cr\cr&&\le \frac{N+1}{2}\|g\|_{\infty} {\rm diam} \,\Omega\int_{S_{q,n}^+} |\langle D^2u(y), (y-z)\rangle|\, dy
\end{eqnarray*}
and passing to the limit as $n\to \infty$, we get
 \begin{eqnarray}\label{final-lim}
 &&{\epsilon}\int_{S_{q}}|\langle D^2u(y), (y-z)\rangle|^2\, dy\cr\cr&&\le \frac{N+1}{2}\|g\|_{\infty} {\rm diam} \,\Omega\int_{S_{q}} |\langle D^2u(y), (y-z)\rangle|\, dy
\end{eqnarray}
where
$$S_{q}:=\{\langle \nabla u(x), x-z\rangle -u(x)-q\ge 0\}.$$
The use of  H\"older's inequality to the integral in the right hand side gives
$$\int_{S_{q}} |\langle D^2u(y), (y-z)\rangle|\, dy\le \left(\int_{S_{q}}|\langle D^2u(y), (y-z)\rangle|^2\, dy\right)^{\frac{1}{2}}\mu(S_q)^{\frac{1}{2}}$$
so that \eqref{final-lim} gives
\begin{eqnarray*}\label{nomen}
 {\epsilon}\left(\int_{S_{q}}|\langle D^2u(y), (y-z)\rangle|^2\, dy\right)^{\frac{1}{2}}\le \frac{N+1}{2}\|g\|_{\infty} {\rm diam}\Omega \cdot\mu(S_q)^{\frac{1}{2}}
\end{eqnarray*}
On the other hand, if  we first use  H\"older's inequality with exponents $\frac{N}{N-1}$ and $N$, with the aid of the Sobolev and H\"older inequalities  we get
\begin{eqnarray*}
&&\int_{S_{q}}|\langle \nabla u(x), x-z\rangle -u(x)-q|\, dx\cr\cr&&\le
\left(\int_{S_{q}}|\langle \nabla u(x), x-z\rangle -u(x)-q|^{\frac{N}{N-1}}\, dx\right)^{\frac{N-1}{N}}\mu(S_q)^{\frac{1}{N}}\cr\cr&&
\le c\left(\int_{S_{q}}|\nabla(\langle \nabla u(x), x-z\rangle -u(x)-q|\right)\mu(S_q)^{\frac{1}{N}}\cr\cr&&=
c\left(\int_{S_{q}}|\langle D^2u(x), (x-z)\rangle|\right)\mu(S_q)^{\frac{1}{N}}\cr\cr&&\le
c\left(\int_{S_{q}}|\langle D^2u(x), (x-z)\rangle|^2\right)^{\frac{1}{2}}\mu(S_q)^{\frac{1}{2}+\frac{1}{N}}
\end{eqnarray*}
and hence, combining the last two inequalities above, we have
\begin{eqnarray}
\int_{S_{q}}|\langle \nabla u(x), x-z\rangle -u(x)-q|\, dx\le \frac{c}{\epsilon}\|g\|_{\infty} \mu(S_q)^{1+\frac{1}{N}}
\end{eqnarray}
At this point, we use Cavalieri principle  obtaining  that
\begin{eqnarray}\label{cavalieri}
\int_{q}^{\infty}\mu\{x:\langle \nabla u(x), x-z\rangle -u(x)\ge t\}\, dt\le \frac{c}{\epsilon}\|g\|_{\infty} \mu(S_q)^{\frac{N+1}{N}}
\end{eqnarray}
that is
\begin{eqnarray}
&&\int_{q}^{\infty}\mu\{x:\langle \nabla u(x), x-z\rangle -u(x)\ge t\}\, dt\cr\cr&&\le \frac{c}{\epsilon}\|g\|_{\infty} \mu\{x:\langle \nabla u(x), x-z\rangle -u(x)\ge q\}^{\frac{N+1}{N}}
\end{eqnarray}
 Arguing now as in Proposition 4.5 in \cite{BB},
 we set 
 $$\Theta(q):=\int_{q}^{\infty}\mu\{x:\langle \nabla u(x), x-z\rangle -u(x)\ge t\}\, dt$$
having the differential inequality
\begin{eqnarray}
\label{disdif}\Theta(q)\le \frac{c}{\epsilon}\|g\|_{\infty} \left(-\Theta'(q)\right)^{\frac{N+1}{N}}
\end{eqnarray}
for every $q>(r+1)\diam \Omega+\|u\|_{\infty}+1=:r_0$.
 
\noindent Recalling a result due to Hartman and Stampacchia (see \cite{HS}), we get
$\Theta(q)=0$  $\forall q>r_0+(N+1)\left(\frac{c}{\epsilon}\|g\|_{L^{\infty}(\Omega)}\right)^{\frac{N}{N+1}}\left(\int_{r_0}^{+\infty}\mu\{x:\langle \nabla u(x), x-z\rangle -u(x)\ge t\}\, dt\right)^{\frac{1}{N+1}}$.
Since inequality \eqref{disdif} gives
\begin{eqnarray*}
&&q_0:=r_0+(N+1)\left(\frac{c}{\epsilon}\|g\|_{L^{\infty}(\Omega)}\right)^{\frac{N}{N+1}}\mu(\Omega)^{\frac{1}{N}}\cr\cr
&&>
r_0+(N+1)\left(\frac{c}{\epsilon}\|g\|_{L^{\infty}(\Omega)}\right)^{\frac{N}{N+1}}\left(\int_{r_0}^{+\infty}\mu\{x:\langle \nabla u(x), x-z\rangle -u(x)\ge t\}\, dt\right)^{\frac{1}{N+1}},
\end{eqnarray*}
 we deduce that $\Theta(q)=0$, that is $|S_q|=0$, $\forall q>q_0$. Then for $z\in \partial \Omega$ fixed and  for a.e $x\in\Omega$, we have
$$\langle \nabla u(x), x-z\rangle -u(x)\le q.$$
Now consider a sequence $\{z_n\}_n\subset \partial \Omega$ such that $\overline{\{z_n\}}_n= \partial \Omega$ and observe that for a.e $x\in\Omega$, we have
\begin{eqnarray}
   \langle \nabla u(x), x-z_n\rangle -u(x)\le q 
\end{eqnarray}
 so that
\begin{eqnarray}\label{stiman}
\langle \nabla u(x), \frac{x-z_n}{|x-z_n|}\rangle \le \frac{\|u\|_{\infty}+q }{|x-z_n|}\le \frac{\|u\|_{\infty}+q }{\dist(x, \partial \Omega)}
\end{eqnarray}
If we choose $\{z_n\}_n$ such that
$$\frac{x-z_n}{|x-z_n|}\to \frac{\nabla u(x)}{|\nabla u(x)|},$$
passing to the limit for $n\to\infty$ in \eqref{stiman}, we obtain
\begin{eqnarray*}
|\nabla u(x)|\le \frac{Q}{\dist(x, \partial \Omega)}\quad \mathrm{for\, a.e.} \quad x\in \Omega
\end{eqnarray*}
 with $Q= \|u\|_{\infty}+q_0$ depending only on  the data of the problem but independent on $\bar \epsilon$.
\end{proof}

\section{Proof of Theorem \ref{mainth}}\label{sec4}

 The aim of this  section is the proof of  Theorem \ref{mainth} splitted in several steps for the convenience of the reader. In the first one 
we shall consider approximating functionals ${\mathcal I_k}(u)$ which are uniformly convex and we shall obtain an estimate for the gradient of their minimizers. The crucial features of the estimate are  its  dependence only on the data of the problem and its uniformity with respect to $k$. 
\begin{proof}
{\it Step 1} 
Arguing similarly as in Theorem \ref{approximation}, we get the existence of  a nondecreasing sequence  $\{f_k\}_k$ of convex, smooth  functions which converges uniformly to $f$ and such that $\forall k\in \N$, it holds
\begin{eqnarray}\label{unif}
f_k(\theta\xi+(1-\theta)\zeta)\le\theta f_k(\xi)+(1-\theta)f_k(\zeta)-{\frac{\epsilon}{4}} \theta(1-\theta)|\xi-\zeta|^2
\end{eqnarray}
for every $\xi, \zeta\in \R^N$ such that $|\zeta|>{r+2}$ and for every $\theta\in [0,1]$. Moreover let us   assume $g_k\in C^{\infty}(\Omega)$ converging  to $g$ in $L^{\infty}$. We will use later that $ k\in \N$ sufficiently large, it holds
\begin{equation}\label{stimag}
\|g_k\|_{L^{\infty}(\R^N)}\le \|g\|_{L^{\infty}(\Omega)}+1   
\end{equation}
Now consider the functional
\begin{eqnarray}
& \displaystyle {\mathcal I_k
(u)}:=\int_{\Om} \Big{[}f_k(\nabla
u(x))+\frac{1}{k}|\nabla u(x)|^2+g_k(x)u(x)\Big{]}\,\,d{x}.
\end{eqnarray}
and denote with $u_k$ the unique minimizer in the class $\phi+W^{1,1}_0(\Omega)$ (see Theorem 8.1 \cite{St} and Theorem 3.1 in \cite{BC} ) and observe that, by Theorem 3.1 in \cite{BC}, it is actually locally Lipschitz.

\noindent Therefore estimate \eqref{estimateapp}  yields for every minimizer $u_k$, that is 
\begin{eqnarray*}
 |\nabla u_k(x)|\le \frac{Q_k}{\dist(x,\partial \Omega)}\quad \mathrm{for\, a.e} \quad x\in  \Omega,  
\end{eqnarray*}
where  
$Q_k=Q_k(\|u_k\|_\infty, R, {\rm diam}\,\Omega, \epsilon, N, L, \|\phi\|_{1,\infty})$ and it is independent on $\frac{1}{k}$.
Next aim is to show that  we can estimate uniformly the constants $Q_k$. It will be enough to  prove that we can estimate uniformly $\| u_k\|_{\infty}$. 

\noindent Define 
$$h_k(\xi)=f_k(\xi)+\frac{1}{k}|\xi|^2$$
and observe that
\[
h_k(\xi)\le f(\xi)+|\xi|^2=h(\xi) \quad \forall \xi\in\R^N, 
\]
since $f_k(\xi)\le f(\xi)$
 and that $$f(\xi)-1\le h_k(\xi)\quad \forall \xi\in\R^N$$
Applying Theorem \ref{thm:comparison1}, we  obtain a constant $U_0$ such that
\begin{eqnarray}\label{star}\|u_k\|_{\infty}\le U_0 \quad \forall k\in \N.\end{eqnarray}
Then, Theorem \ref{aprioriestimate} gives
\begin{eqnarray}\label{costante}
|\nabla u_k(x)|\le \frac{Q}{\dist(x, \partial \Omega)}\quad \mathrm{for\, a.e.} \quad x\in \Omega
\end{eqnarray}
with $Q$ independent on $k$.

\medskip

{\it Step 2} 
Let us prove that  $\{u_k\}_k$ is a minimizing sequence for the functional $\mathcal{I}$ and that it converges weakly in $W^{1,1}(\Omega)$ to a minimum which is locally Lipschitz. 

\noindent Let   $u \in \phi+W^{1,1}_0$ be a minimizer of  $\mathcal{I}$ and observe that, by the minimality of $u_k$ for $\mathcal{I}_k$,  we have that
\begin{eqnarray*}
\int_{\Omega}f_k(\nabla u_k)+\frac{1}{k}|\nabla u_k|^2+g_k(x)u_k(x)\, dx\le \int_{\Omega}f_k(\nabla u)+\frac{1}{k}|\nabla u|^2+g_k(x)u(x)\, dx.
\end{eqnarray*}
Since the dominated convergence Theorem yields
\begin{eqnarray*} \lim_{k}\int_{\Omega}f_k(\nabla u)+\frac{1}{k}|\nabla u|^2+g_k(x)u(x)\, dx=\int_{\Omega}f(\nabla u)+g(x)u(x)\, dx,
\end{eqnarray*}
 for every $\sigma>0$ and $k\in \N$ sufficiently large, we get
\begin{eqnarray}\label{disconeps}
\int_{\Omega}\!\!f_k(\nabla u)+\frac{1}{k}|\nabla u|^2+g_k(x)u(x)\, dx\le \!\!\int_{\Omega}\!\!f(\nabla u)+g(x)u(x)\, dx+\sigma.
\end{eqnarray}
On the other hand, Theorem \ref{approximation} ensures that for $k\in \N$ sufficiently large, we have
\begin{eqnarray*}
&&\int_{\Omega}f_k(\nabla u_k)+\frac{1}{k}|\nabla u_k|^2+g_k(x)u_k(x)\, dx\cr\cr&&
\ge \int_{\Omega}f(\nabla u_k)-\sigma+ (g_k(x)-g(x))u_k(x)+g(x)u_k(x)\, dx.
\end{eqnarray*}
Since $g_k$ converges uniformly to $g$ and  \eqref{star} holds, for a sufficiently large  $k\in \N$, we have also that
\begin{eqnarray}\label{sotto}
&&\int_{\Omega}f_k(\nabla u_k)+\frac{1}{k}|\nabla u_k|^2+g_k(x)u_k(x)\, dx\cr\cr&&
\ge \int_{\Omega}f(\nabla u_k)+g(x)u_k(x)\, dx- c(\mu(\Omega), U_0)\, \sigma.
\end{eqnarray}
Combining  \eqref{sotto} with\eqref{disconeps}, we obtain
\begin{eqnarray}
&&\int_{\Omega}f(\nabla u_k)+g(x)u_k(x)\, dx\cr\cr&&\le \int_{\Omega}f(\nabla u)+g(x)u(x)\, dx+\sigma(1+c(\mu(\Omega)), U_0)). 
\end{eqnarray}
 Since $u_k$ is minimizing, we deduce the existence of a subsequence $u_{k_j}$ weakly convergent in $W^{1,1}(\Omega)$ to a function $\bar u$ that is a minimizer of $\mathcal{I}$ in $\phi+W_0^{1,1}(\Omega)$. Let us define 
$$\Omega_n=\Big\{x\in \Omega:\dist(x, \partial \Omega)>\frac{1}{n}\Big\} $$ and consider $u_{k_j}$ restricted to $\Omega_n$. By \eqref{costante}, we get
a uniform estimate of $u_{k_j}$ in $W^{1,\infty}(\Omega_n)$ and hence,  up to a subsequence,  $u_{k_j}$ weakly-$*$ converges to $\tilde u$ in  $W^{1,\infty}(\Omega_n)$ and hence weakly in $W^{1,1}(\Omega_n)$. It follows that $\tilde u=\bar u$ in $\Omega_n$ and, since $\tilde u$ satisfies the same estimate of $u_{k_j}$, we get that $\bar u$ is locally Lipschitz.

\medskip

{\it Step 3} It remains to prove that any minimizer $u\in \phi+W^{1,1}_0(\Omega)$ of the functional $\mathcal{I}$ is locally Lipschitz too. 

\noindent Let us denote
$$A=\{x\in \Omega: \nabla u(x), \nabla \bar u(x) \in B_{r+1}\}$$and
$$B=\{x\in A^c: \nabla u(x)\ne \nabla \bar u(x)\}.$$
We  prove that $\mu(B)=0$. Assume that $\mu(B)>0$, it follows that we can write
\begin{eqnarray}
&&\mathcal{I}\left(\frac{1}{2} u(x)+\frac{1}{2} \bar u(x)\right)\cr\cr &&=\int_B f\left(\frac{1}{2}\nabla u(x)+\frac{1}{2}\nabla \bar u(x)\right)+\frac{1}{2}g(x)(u(x)+\bar u(x))\, dx\cr\cr&&
+\int_{B^c} f\left(\frac{1}{2}\nabla u(x)+\frac{1}{2}\nabla \bar u(x)\right)+\frac{1}{2}g(x)(u(x)+\bar u(x))\, dx\cr\cr&&
<\frac{1}{2}\mathcal{I}(u)+\frac{1}{2}\mathcal{I}(\bar u)\le \frac{1}{2}\mathcal{I}(u)+\frac{1}{2}\mathcal{I}(u)=\mathcal{I}(u)
\end{eqnarray}
where, in the second last line, we used the strict convexity of the function $f$. Indeed, for every  $x\in{B}$, at least one  between $\nabla u(x)$ and $\nabla \bar u(x)$ belongs to the set in which $f$ is strictly convex. The minimality of $u$ gives a contradiction.
\end{proof}
 With the next result  we observe the validity of H\"older continuity property up to the boundary of the minimizers of $\mathcal{I}$ provided we add some assumptions.  

  \begin{corollary}\label{Clarke}
 Let $\Omega$, $f$ and $g$ be such that  the hypotheses of Theorem \ref{mainth} hold. In addition, assume that $\Omega$ has the boundary of class $ C^{1,1}$ and that $f$ satisfies the following coercivity condition 
 $$f(\xi)\ge c |\xi|^p$$
 for some constants $c>0$ and $p>\frac{n+1}{2}$.
 For a fixed function $\phi$ satisfying the LBSC, any minimizer $u$ of $\mathcal{I}$ in \eqref{funzionale} is  H\"older continuous on $\bar \Omega$ of order
 $$\alpha:=\frac{2p-n-1}{4p+n-3}.$$
 \end{corollary}
 The proof follows exactly the one given in \cite{Cl}, see also Theorem 4.2 in \cite{BC}. Here we give only some details. First of all, we observe that the assumption on the boundary of $\Omega$ implies that for every point $\gamma$ of $\partial \Omega$ there exists a closed ball of  radius $R>0$ contained in $\bar \Omega$  which contacts $\partial \Omega$ at $\gamma$. Hence every point in {$\Omega$} of distance less than $R$ from {$\partial\Omega$} admits a unique projection on  $\partial \Omega$.  Moreover, for every $\gamma\in \partial \Omega$, there is a unique exterior   unit normal $\nu(\gamma)$ to 
$\partial \Omega$ and the function $\nu(\gamma)$ is Lipschitz continuous.
The H\"older continuity estimate is proven first for  $x\in \Omega$ sufficiently close to the boundary and $\gamma\in \partial \Omega$ projection of $x$ onto $\partial \Omega$ and {then} globally for every $x$ and $y$ in $\bar \Omega$ .

 \section{The nonconvex case}\label{nccase}
 In this Section we prove, as an application of Theorem \ref{mainth}, an existence result of a locally Lipschitz minimizer for the functional $\mathcal{I}$ allowing $f$ to be nonconvex. 

\begin{theorem}\label{nonconvex}
 Let $\Omega$  and $g$ be such that  the hypotheses of Theorem \ref{mainth} hold. Moreover assume that $|\{x\in \Omega: g(x)=0\}|=0$ and that $f^{**}$ satisfies  the
following hypotheses
\begin{itemize}
\item[($F^*$1)]  $f^{**}(\xi)\ge 0$ and    $f^{**}(0)=0$;
\item[($F^*$2)]  
${\rm dom}f^{**}:=\{x\in \R^n: f(x)\in \R\}
=\R^n$;
\item[($F^*$3)] there exists a radius $r>0$  such that
$$f^{**}(\theta\xi+(1-\theta)\zeta)\le\theta f^{**}(\xi)+(1-\theta)f^{**}(\zeta)-\frac{\epsilon}{2} \theta(1-\theta)|\xi-\zeta|^2$$
 for every $\xi, \zeta\in \R^N$ such that $|\zeta|>r+1$ and for every $\theta\in [0,1]$.
\item[($F^*$4)]   
 $\{\xi \in \R^N:f(\xi)\neq f^{**}(\xi)\}=\cup_{i=1}^\infty {\rm int }\, S_i$, where the sets $S_i$ are convex and such that ${\rm int}\,S_i\neq \emptyset$, $\bar S_i\cap \bar S_j=\emptyset$.
 \end{itemize}
 Then there exists a minimizer of $\mathcal{I}$ that is locally Lipschitz continuous in $\Omega$.
\end{theorem}
\begin{proof} Let us consider the functional
$$\mathcal{I^{**}}( u)=\int_{\Omega} f^{**}(\nabla  u)+g(x) u\, dx$$
It follows by Theorem \ref{mainth} that every $ u$  minimizer of $\mathcal{I^{**}}$ is locally Lipschitz continuous in $\Omega$ and hence a.e. differentiable. 

\noindent Let us show that the set $\tilde \Omega~:=~\{x\in \Omega: \nabla  u(x)\in \cup_{i=1}^{\infty}{\rm int}\,S_i \text{ and } g(x)\neq 0
\}$ is negligible. In contradiction, we assume that $\tilde \Omega$ has a positive measure. It follows from the existence of $x\in \Omega$ of differentiability for $ u$ such that $\nabla  u(x)\in \cup_{i=1}^{\infty}{\rm int}\,S_i$ and, without loss of generality, we can assume that $x$ is a density point $1$ for the set $\{x\in \Omega:g(x)>0\}$.

\noindent By a pyramidal construction of the type in \cite{STC}, we will show that for every $\bar x\in\tilde\Omega$, there exists a ball centered in $\bar x$ where the function $u$ can be modified to obtain a function $\tilde u_{\rho}$ that is in $\phi+W^{1,1}_0(\Omega)$ and is still locally Lipschitz continuous and such that
$\mathcal{I^{**}}( \tilde u_\rho)< \mathcal{I^{**}}( u)$ gets the incongruity.

\noindent \textit{ Step 1} \quad Fix $\bar x\in \tilde\Omega$ and assume that $\nabla  u(\bar x)\in  \text{int}\,S_i$, for some $i\in \N$. Now define $\delta:=\text {dist}(\nabla u(\bar x),\partial S_i )$ and $\delta':=\sup\{|\nabla u(\bar x)-\xi|,\, \xi\in\partial S_i\}$ and

\noindent observe that the differentiability of $u$ at $\bar x$ implies that there exists $\rho>0$ such that $B_{\rho}(\bar x)\subset \Omega$ and

\begin{eqnarray}\label{prima11}
u(x)\le u(\bar x)+\nabla u(\bar x)\cdot(x-\bar x)+\frac{\delta}{2}|x-\bar x|\quad \forall x\in B_{\rho}(\bar x)
\end{eqnarray}
and
\begin{eqnarray}\label{prima112}
u(x)\ge u(\bar x)+\nabla u(\bar x)\cdot(x-\bar x)-\frac{\delta'}{2}|x-\bar x|\quad \forall x\in B_{\rho}(\bar x).
\end{eqnarray}
Set $\mathcal{S}_i:=\!S_i-\nabla u(\bar x)$ so that $0\!\in \!\text{int}\,\mathcal{S}_i$, $\delta\!=\!\text {dist}(0,\partial \mathcal{S}_i)$ and $\delta'\!= \!\sup\{ |\xi|:\xi \in \partial \mathcal{S}_i\}$. The convexity of $\mathcal{S}_i$ allows to choose  $\xi_1, \ldots, \xi_{n+1}\in \partial\mathcal{S}_i=\partial S_i-\nabla u(\bar x)$ such that
$$0=\sum_{j=1}^{n+1}\lambda_j\xi_j \quad \text{with}\,\, \sum_{j=1}^{n+1}\lambda_j=1, \lambda_j>0.$$ 

\noindent At this point, assume  $g(\bar x)>0$  and define the functions
\begin{eqnarray}\label{seconda1}
v(x):=\sup_j \xi_j\cdot x \quad \forall x\in \R^n
\end{eqnarray}
and
\begin{eqnarray}\label{terza1}
w_\rho(x):=v(x-\bar x)+u(\bar x)+\nabla u(\bar x)(x-\bar x) -\frac{\delta}{3}\rho  \quad \forall x\in \Omega.
\end{eqnarray}
Since
$\delta |x|\le v(x)$, 
it follows that, for every $ x\in\partial B_{\rho}(\bar x)$, we have
\begin{eqnarray}\label{seconda11}
w_{\rho}(x)&\ge& u(\bar x)+\nabla u(\bar x)(x-\bar x)+\delta \rho -\frac{\delta}{3}\rho\cr\cr
&\ge& u(\bar x)+\nabla u(\bar x)(x-\bar x)+\frac{\delta}{2} \rho\ge u(x).
\end{eqnarray}
On the other hand, taking into account that $v(x)\le \delta' |x|$ , we get
\begin{eqnarray}\label{seconda112}
w_{\rho}(x)&\le& u(\bar x)+\nabla u(\bar x)(x-\bar x)+\delta'|x-\bar x|-\frac{\delta}{3}\rho \cr\cr &\le& u(\bar x)+\nabla u(\bar x)(x-\bar x)-\frac{\delta'}{2}|x-\bar x|\le u(x)
\end{eqnarray}
for every $ x\in B_{\frac{\delta}{\delta'}\frac{2\rho}{9}}(\bar x)$.
 Moreover,  if we set $$ E_\rho:=\!\{x\in\! \Omega\!: \!w_{\rho}(x)\!\le \!u(x)\}\cap B_{\rho}(\bar x)$$ and define
\begin{eqnarray*}
\tilde u_{\rho}(x):=\begin{cases} w_{\rho}(x)& \text{ if } x\in E_{\rho}\\
u(x)& \text{ if } x\in \Omega\setminus E_{\rho},
\end{cases}
\end{eqnarray*}
we have $\tilde u_{\rho}(x)\le u(x)$ in $\Omega$ and $\tilde u_{\rho}=\phi$ on $\partial \Omega$.

\noindent In the case $g(\bar x)<0$, we define the functions
\begin{eqnarray*}
v(x):=\inf_j \xi_j\cdot x \quad \forall x\in \R^n
\end{eqnarray*}
and
\begin{eqnarray}\label{terza11}
w_{\rho}(x):=v(x-\bar x)+u(\bar x)+\nabla u(\bar x)(x-\bar x) +\frac{\delta}{3}\rho  \quad \forall x\in \Omega 
\end{eqnarray} and, arguing similarly we have done above,  we obtain a function $\tilde u_\rho(x)\ge u(x)$.

\medskip

\noindent {\it Step 2} \quad Now we show that 
\begin{eqnarray}\label{disminimo}
\mathcal I^{**}(\tilde u_\rho)< \mathcal I^{**}( u)
\end{eqnarray}
{that will be true as soon as }
\begin{eqnarray}\label{claim}
\int_{E_\rho} f^{**}(\nabla  \tilde u_\rho)+g(x) \tilde u_\rho\, dx< \int_{E_\rho} f^{**}(\nabla   u)+g(x)  u\, dx.
\end{eqnarray}
We will show that
\begin{eqnarray}\label{claim1}
\int_{E_\rho} f^{**}(\nabla  \tilde u_\rho)\, dx\le \int_{E_\rho} f^{**}(\nabla   u)\, dx.
\end{eqnarray}
and
\begin{eqnarray}\label{claim2}
\int_{E_\rho} g(x) \tilde u_\rho\, dx< \int_{E_\rho} g(x)  u\, dx.
\end{eqnarray}
To this aim  we observe that ($F^*4$) implies that
$f^{**}(\xi)=a_i\cdot \xi+b_i$ for every $\xi\in S_i$ and denote by $\nabla f^{**}$ a selection of the subgradient of $f^{**}$ such that $\nabla f^{**}(\xi)=a_i$  for every $\xi\in\partial S_i$.

\noindent Consider again the ball $B_\rho(\bar x)$ fixed above.
  The convexity implies that
\begin{eqnarray}\label{dallaconv}
&&\int_{B_\rho(\bar x)} f^{**}(\nabla  \tilde  u_\rho)\, dx=\int_{B_\rho(\bar x)\setminus E_\rho} f^{**}(\nabla    u)\, dx+\int_{ E_\rho} f^{**}(\nabla    w_\rho)\, dx\cr\cr&&\le\int_{B_\rho(\bar x)\setminus E_\rho} f^{**}(\nabla    u)\, dx+
\int_{E_\rho} f^{**}(\nabla    u)+\nabla f^{**}(\nabla w_\rho)( \nabla    w_\rho-\nabla u) \, dx.
\end{eqnarray}
If we set
$$\tilde v_\rho(x):=\min\{w_\rho(x)-u(x), 0\}\quad \text{in}\, B_\rho(\bar x),$$
it follows that $\tilde v_\rho\in W^{1,1}_0(B_\rho(\bar x))$,
\begin{eqnarray*}
\tilde v_\rho=\begin{cases}
w_\rho-u \quad \text{in}\,\, E_\rho\\
0\quad \quad \quad\text{in} \,\, B_\rho(\bar x)\setminus E_\rho,
\end{cases}
\end{eqnarray*}  and
\begin{eqnarray*}
\nabla \tilde v_\rho=\begin{cases}
\nabla w_\rho-\nabla u \quad \text{in}\,\, E_\rho\\
0\quad \quad \quad \quad \quad\text{in} \,\, B_\rho(\bar x)\setminus E_\rho.
\end{cases}
\end{eqnarray*}
Therefore, since
$f^{**}(\xi)=a_i\cdot \xi+b_i \quad \forall\xi\in \bar S_i$ and $\nabla w_\rho\in \partial S_i$ a.e. $x\in B_\rho(\bar x)$,  we get
\begin{eqnarray*}
\int_{ E_\rho} \nabla f^{**}(\nabla w_\rho) (\nabla    w_\rho-\nabla u)\, dx=\int_{ B_\rho(\bar x)} \nabla f^{**}(\nabla w_\rho) \nabla \tilde v_\rho\, dx=\int_{ B_\rho(\bar x)} a_i\nabla \tilde v_\rho\, dx=0.
\end{eqnarray*}
and inequality \eqref{dallaconv} reduces to
\begin{eqnarray*}
\int_{B_\rho(\bar x)}\!\!\!\! f^{**}(\nabla  \tilde  u_\rho)\, dx
\le\int_{B_\rho(\bar x)\setminus E_\rho} f^{**}(\nabla    u)\, dx+\!\!\int_{ E_\rho} f^{**}(\nabla u)  \, dx
\end{eqnarray*}
that is
\begin{eqnarray*}
\int_{ B_\rho(\bar x)\setminus E_\rho}f^{**}(\nabla   \tilde  u_\rho)\, dx+\int_{  E_\rho}f^{**}(\nabla   \tilde  u_\rho)\, dx
\le \int_{ B_\rho(\bar x)\setminus E_\rho}f^{**}( \nabla  u) dx+
\int_{ E_\rho} f^{**}(\nabla    u)\, dx
\end{eqnarray*}
 which {yields} \eqref{claim1}.

 \noindent Let us assume that $g(\bar x)>0$ and recall that, in this case, we found $\tilde u_{\rho}(x)\le u(x)$ in $\Omega$. In order to prove \eqref{claim2}, we assume,  without loss of generality, $\bar x$ as a Lebesgue point for $g(x)$ and assume that it is of density $1$ for the set  
 $$\Omega_{\bar x}:=\left\{x\in \Omega: g(x)\ge \frac{g(\bar x)}{2}\right\}.$$
 It follows that, for a fixed $\sigma>0$ there exists $\rho(\sigma)>0$ such that
\begin{eqnarray}\label{uno}|B_\rho(\bar x)\cap\Omega_{\bar x}|\ge (1-\sigma)|B_\rho(\bar x)|
\end{eqnarray}
and
\begin{eqnarray}\label{due}|B_\rho(\bar x)\setminus \Omega_{\bar x}|\le \sigma |B_\rho(\bar x)| \end{eqnarray}
 for every $\rho<\rho(\sigma)$. From now on, we shall assume $\rho$ such that \eqref{prima11}, \eqref{prima112}, \eqref{uno} and \eqref{due} hold.

 \noindent Showing  inequality \eqref{claim2} is obviously equivalent to prove 
 \begin{eqnarray}\label{scritequiv}
&&\int_{ E_\rho} g(x)( u-\tilde u_\rho)   \, dx=\int_{B_\rho(\bar x)} g(x)( u-\tilde u_\rho)   \, dx\ge 0.
\end{eqnarray}
Observe that
 \begin{eqnarray*}
&&\int_{B_\rho(\bar x)} g(x)( u-\tilde u_\rho)   \, dx=\int_{B_\rho(\bar x)\cap \Omega_{\bar x}} g(x)( u-\tilde u_\rho)   \, dx+\int_{B_\rho(\bar x)\setminus \Omega_{\bar x}} g(x)( u-\tilde u_\rho)   \, dx\cr\cr
&&\ge \int_{B_{\frac{\delta}{\delta'}\frac{\rho}{9}}(\bar x)\cap \Omega_{\bar x}} g(x)( u-\tilde u_\rho)   \, dx+\int_{B_\rho(\bar x)\setminus \Omega_{\bar x}} g(x)( u-\tilde u_\rho)   \, dx.
\end{eqnarray*}
Since analogous calculations to those we have done in \eqref{seconda112} yield
$ u-\tilde u_\rho\ge \frac{\delta}{6}\rho$ on the ball $B_{\frac{\delta}{\delta'}\frac{\rho}{9}}$,  we can deduce  that 
$$\int_{B_{\frac{\delta}{\delta'}\frac{\rho}{9}}(\bar x)\cap \Omega_{\bar x}} g(x)( u-\tilde u_\rho)   \, dx\ge \frac{g(\bar x)}{2}\frac{\delta}{6}\rho \, |B_{\frac{\delta}{\delta'}\frac{\rho}{9}}(\bar x)\cap \Omega_{\bar x}|.$$
Moreover, recalling the definition of $u_\rho$, the inequality \eqref{prima11}, the definition of $w_\rho$ and the fact that $v(x)\ge \delta |x|\ge 0$, we have
$$u(x)-\tilde u_{\rho}(x)=(u(x)-w_{\rho}(x))\chi_{E_\rho}(x)\le \frac{\delta}{2}\rho-v(x-\bar x)+\frac{\delta}{3}\rho\le \delta \rho.$$
It follows that
 \begin{eqnarray*}
&&\int_{B_\rho(\bar x)} g(x)( u-\tilde u_\rho)   \, dx\ge
\frac{g(\bar x)}{2}\frac{\delta}{6}\rho \, |B_{\frac{\delta}{\delta'}\frac{\rho}{9}}(\bar x)\cap \Omega_{\bar x}|-\alpha \delta \rho\, |B_\rho(\bar x)\setminus \Omega_{\bar x}|\cr\cr&&
\ge \frac{g(\bar x)}{2}\frac{\delta}{6}\rho \, (1-\sigma)|B_{\frac{\delta}{\delta'}\frac{\rho}{9}}(\bar x)|-\alpha \delta \rho\, \sigma |B_\rho(\bar x)|
\cr\cr
&&=-\left( \frac{g(\bar x)}{2}\frac{\delta}{6}\,|B_{\frac{\delta}{\delta'}\frac{\rho}{9}}(\bar x)|+ \alpha \delta\, |B_\rho(\bar x)|\right)\rho\sigma +\frac{g(\bar x)}{2}\frac{\delta}{6}\,|B_{\frac{\delta}{\delta'}\frac{\rho}{9}}(\bar x)|\, \rho\cr\cr
&&=\omega_n \rho^{n+1}\delta\left[-\left( \frac{g(\bar x)}{2}\frac{1}{6}\,\left({\frac{\delta}{\delta'}\frac{1}{9}}\right)^n+ \alpha \right)\sigma +\frac{g(\bar x)}{2}\frac{1}{6}\left(\frac{\delta}{\delta'} \frac{1}{9}\right)^n\, \right],
\end{eqnarray*}
where we denoted by $\omega_n$ the measure of the unit ball in $\R^n$ and we also used \eqref{uno} and \eqref{due}. Clearly, for $\sigma $ sufficiently small, the right-hand side of the previous inequality is positive and hence \eqref{scritequiv} holds. 

\noindent  In case $g(\bar x)<0$, the previous arguments are valid taking
$$\tilde v_\rho(x):=\max\{w_\rho(x)-u(x), 0\}\quad \text{in}\, B_\rho(\bar x).$$
We conclude that inequality \eqref{claim} holds.
\end{proof}

 \begin{corollary}\label{cor:gsign}
 Let $f$ be such that $f^{**}$ satisfies the hypotheses (F$^*$1)-(F$^*$4). Assume that $\Omega$ is a $R$-uniformly convex set and that $g\colon \Omega\to \R$ is a measurable bounded function that doesn't change sign.   For a fixed function $\phi$ satisfying the LBSC, any minimizer $u$ of $\mathcal{I}$  is  locally Lipschitz in $\Omega$.
 \end{corollary}
 \begin{proof}
It is enough to reason similarly we have done in the proof of Theorem \ref{nonconvex}. We start with $u$ a minimizer of $\mathcal{I}^{**}$ that is locally Lipshitz continuous thanks to Theorem \ref{mainth}. Consider the set $\Omega_0=\{x\in\Omega: \nabla u(x)\in\cup_{i\in\N}{\rm int}S_i\}$  in place of $\tilde \Omega$ and assume that $\Omega_0$ has positive measure. Without loss of generality, we assume $g(x)\ge 0$.

\noindent Arguing as in {\it Step 1} of Theorem \ref{nonconvex}, we can prove that for every $\bar x\in\Omega_0$ there exists a ball centered in $\bar x$ where  the function $u$ can be modified  to obtain a function $\tilde u$ that is still locally Lipschitz and $\tilde u\le u$. We notice that  {\it Step 2}  becomes easier because inequality 
\begin{eqnarray}
\int_{E_\rho} g(x) \tilde u\, dx\le \int_{E_\rho} g(x)  u\, dx.
\end{eqnarray}
is an  immediate consequence of the constant sign of  $g$. Recalling \eqref{claim1}, it follows that   
\begin{eqnarray}\label{con*}
\int_{E_\rho} f^{**}(\nabla \tilde u)+ g(x) \tilde u\, dx\le \int_{E_\rho}f^{**}(\nabla  u)+ g(x)  u\, dx.
\end{eqnarray}

Next aim is to prove the existence of another minimizer $\bar u$ such that $\nabla \bar u(x)\in \R^n\setminus \cup_{i}\,\rm{int }S_i$ for a.e.  $x$ in which $\bar u$ is differentiable.

\noindent We first observe that inequalities \eqref{prima11} and \eqref{prima112} are still valid in  every ball $B_{\rho_h}(\bar x)$ with $\rho_h\le \rho$. Hence defining $w_{\rho_h}$  as in \eqref{terza1}, we  obtain  that the closed corresponding sets
$E_{\rho_h}$
cover $ \Omega_0$ in the sense of Vitali. Indeed, analogous estimates to \eqref{seconda11} and \eqref{seconda112} with  $\rho_h$ in place of $\rho$ yield $$\frac{|E_{\rho_h}|}{|B_{\rho_h}|}\ge \frac{|B_{\frac{\delta}{\delta'}\frac{2\rho_h}{9}}|}{|B_{\rho_h}|}=\left(\frac{2\delta}{9\delta'}\right)^n.$$
We deduce that there exists a countable family of points $\{x_m\}$ in $\Omega_0$ and of pairwise disjoint sets $E_{\rho_m}{\subset B_{\rho_m}(x_m)}$  such that
$$|\Omega_0\setminus \cup_m {E_{\rho_m}}|=0.$$
Now define
\begin{eqnarray*}
\bar u=\begin{cases}
u& \text{in}\,\, \Omega\setminus  \Omega_0\\
{w_{\rho_m}} &\text{in} \,\,  {E_{\rho_m}}.
\end{cases}
\end{eqnarray*}
It follows that $\bar u$ is locally Lipschitz continuous. 
 Moreover 
\begin{eqnarray}\label{minimo}
&&\mathcal{I}^{**}(\bar u)=\int_{\Omega} f^{**}(\nabla \bar u)+g(x)\bar u\, dx\cr\cr
&&=\int_{\Omega\setminus  \Omega_0} f^{**}(\nabla  u)+g(x) u\, dx+\sum_m\int_{{E_{\rho_m}}} f^{**}(\nabla {w_{\rho_m}})+g(x){w_{\rho_m}}\, dx\cr\cr
&&\le\int_{\Omega\setminus \Omega_0} f^{**}(\nabla  u)+g(x) u\, dx+\sum_m\int_{{E_{\rho_m}}} f^{**}(\nabla u)+g(x)u\, dx\cr\cr
&&= \mathcal{I}^{**}(u)
\end{eqnarray}
where we used that $\bar u={w_{\rho_m}}= {\tilde u_{\rho_m}}$ in ${E_{\rho_m}}$ and inequality \eqref{con*}.
The minimality property of $u$  implies that the inequality in \eqref{minimo} is actually an  equality and $\bar u$ is a minimizer of $\mathcal{I}^{**}$ such that $\nabla \bar u\in \R^n\setminus\cup_i {\rm i nt} S_i$ almost everywhere. Moreover,   
assumption ($F^*$4) implies $\mathcal{I}(\bar u)=\mathcal{I}^{**}(\bar u)$, so that $$\mathcal{I}(\bar u)=\mathcal{I}^{**}(\bar u)\le \mathcal{I}^{**}( u)\le \mathcal{I}( u)$$ 
that is $\bar u$ is a locally Lipschitz minimizer of  $\mathcal{I}$. 

\noindent Suppose now that $v$ is another minimizer of  $\mathcal{I}$. It holds that
$$\mathcal{I}(v)=\mathcal{I}(\bar u)=\mathcal{I}^{**}(\bar u)\le \mathcal{I}^{**}(v)\le \mathcal{I}(v)$$
that is $$\mathcal{I}^{**}(v)=\mathcal{I}^{**}(\bar u)$$
and hence $v$ is a minimizer for $\mathcal{I}^{**}$. At this point, Theorem \ref{mainth} leads to the conclusion.
\end{proof}

\medskip 

 {\it Acknowledgements.}  The authors are supported by GNAMPA Project 2024 ``Fenomeno di Lavrentiev, Bounded Slope Conditon e  regolarità per minimi di funzionali integrali con crescite non standard e lagrangiane non uniformemente convesse'' (Code CUP-E53C22001930001). F.~Giannetti is supported by GNAMPA Project 2025 (Code CUP E5324001950001) ``Regolarità di soluzioni di equazioni paraboliche a crescita nonstandard degeneri''  and  G. Treu is supported by Unipd DOR project  ``Equazioni differenziali nonlineari, problemi variazionali e controllo'' (Code: DOR2340044).

\medskip 

\noindent {\it Data availability statement.} The authors declare that there is no data used to support the findings of the presented results.

\medskip

\noindent {\it Conflict of interest} The authors are not aware of any Conflict of interest.

\end{document}